\documentclass[12pt]{article}

\usepackage{setspace}
\usepackage{a4wide}

\usepackage{amsthm, amssymb, amstext}
\usepackage[fleqn]{amsmath}
\usepackage{latexsym}
\usepackage[dvips]{graphicx}
\usepackage{comment}
\usepackage{hyperref}
\usepackage{mathtools}
\usepackage{enumerate} 

\usepackage{todonotes}

\newtheorem{theorem}{Theorem}
\newtheorem{lemma}{Lemma}

\newtheorem{claim}{Claim}

\theoremstyle{definition}

\usepackage{amsthm}

\title{A note on  Matching-Cut in $P_t$-free Graphs}
\author{Carl Feghali\thanks{Univ Lyon, EnsL, CNRS, LIP, F-69342, Lyon Cedex 07, France, email: \texttt{carl.feghali@ens-lyon.fr} }}

\date{}

\begin{document}
\maketitle

\begin{abstract}
A matching-cut of a graph is an edge cut that is a  matching. The problem \textsc{matching-cut} is that of recognizing graphs with a matching-cut and is NP-complete, even if the graph belongs to one of  a number of classes. We initiate the study of \textsc{matching-cut} for graphs without a fixed path as an induced subgraph. We show that \textsc{matching-cut} is in P for $P_5$-free graphs, but that there exists an integer $t > 0$ for which it is NP-complete for $P_{t}$-free graphs.  
 \end{abstract}
 
For a connected graph $G$ and a subset $E' \subset E(G)$, we say that $E'$ is a \emph{cutset} if $G - E'$ (i.e., the graph obtained by removing the edges in $E'$ but not their endpoints from $G$)  is disconnected. 

 In 1969, R. L. Graham \cite{grandstrand:2004} defined a cutset of edges to be a \emph{matching-cut} if no two edges
in the cutset have a vertex in common, and studied those graphs which  have no matching-cut, but whose every proper subgraph has a matching-cut. It was Chv\'atal \cite{chvatal} who initiated the study of \textsc{matching-cut}, the complexity problem of recognizing graphs admitting a matching-cut, showing that it is NP-complete, even for graphs with maximum degree at most four, yet in P for graphs with maximum degree at most three (unaware of Chv\'atal's result, Dunbar et al. \cite{dunbar1995nearly} formulated \textsc{matching-cut}, leaving its complexity as an open problem that was repopularized in 2016 in \cite{hedetniemi2016my}). The NP-hardness of \textsc{matching-cut} has since been shown to also hold for graphs with additional or other structural assumptions; see, for example, \cite{bonsma2009complexity, chen2021matching, le2019complexity, randerath2003stable}. To keep this paper short, we refer the reader to \cite{chen2021matching, le2019complexity} and references therein for a thorough discussion, including real-world applications.

For a positive integer $t$, we denote by $P_t$ the induced path with $t$ vertices. A graph $G$ is said to be $P_t$-free if it contains no $P_t$ as an induced subgraph. In this paper, we initiate the study of \textsc{matching-cut} for graphs without a fixed path as an induced subgraph. In particular, the following theorems are proved. 

\begin{theorem}\label{thm:p5}
\textsc{matching-cut} is polynomial-time solvable in $P_{5}$-free graphs.
\end{theorem}

\begin{theorem}\label{thm:p12} There exists an integer $t > 0$ such that \textsc{matching-cut} is NP-complete in $P_{t}$-free graphs.   
\end{theorem} 

Theorem \ref{thm:p5} generalizes a result of Bonsma \cite{bonsma2006sparse}, stating that \textsc{matching-cut} is polynomial-time solvable for cographs.  The proof of Theorem \ref{thm:p5} is short and simple and inspired by the proof of \cite[Theorem 5.4]{bouquet2020complexity}.

The proof of Theorem \ref{thm:p12} is also rather short and simple, and involves new arguments.

\section{The proof of Theorem \ref{thm:p5}}\label{sec:1}

We require the following powerful theorem due to Bacs\'o and Tuza \cite{bacso1990dominating}.

\begin{theorem}\label{p5}
A connected $P_5$-free graph $G$ contains a dominating set $X$ that is either a clique or a $P_3$. Moreover, $X$ can be found in polynomial time.  
\end{theorem}

We call a graph $G$ \emph{full} if it contains a dominating set that is a clique of size $\geq 3$. We also call a coloring of the vertices of $G$ with colors red and blue   \emph{good} if every red vertex is adjacent to at most one blue vertex and every blue vertex is adjacent to at most one red vertex. We say that a good coloring is \emph{strong} if it uses both colors. Note that a good coloring defines a matching-cut if and only if it is strong.

We can easily deal with the case when the graph is not full. We abbreviate red and blue by $r$ and $b$ respectively. 

\begin{lemma}\label{lem:p5good}
Enumerating all good colorings of a non-full connected $P_5$-free graph can be done in polynomial-time. 
\end{lemma}

\begin{proof}
Let $G$ be non-full $P_5$-free graph. Since $G$ is non-full, by Theorem \ref{p5} we can find in polynomial-time a dominating set $X$ of $G$ that is $K_1$, $K_2$ or $P_3$. %To prove the lemma, it suffices to enumerate in polynomial-time all good colorings of $G$. 

For each good coloring $c$ of $X$ and each vertex $x \in X$, how many ways are there of extending $c$ to a good coloring of $X \cup N_{G - X}(x)$? Since at most one vertex adjacent to $x$ can receive the color in $\{r, b\} \setminus c(x)$ at most $|N_{G - X}(x)| + 1$ such ways are possible. Therefore, the number of ways there are of extending $c$ to a good coloring of $G = G[X \cup \bigcup_{x \in X} N_{G - X}(x)]$ is polynomial in the number of vertices. This implies the lemma. 
\end{proof}

We are now ready to prove the theorem. 

\begin{proof}[Proof of Theorem \ref{thm:p5}]
Let $G$ be a $P_5$-free graph. We can assume that $G$ is connected  (since otherwise we apply our algorithm component-wise). By Theorem \ref{p5}, we can find in polynomial-time a dominating set $X$ in $G$ that is either $K_1$, $K_2$, a clique on at least three vertices or $P_3$. If $X$ is $K_1$, $K_2$ or $P_3$, then we apply  Lemma \ref{lem:p5good}.  

Otherwise, since $X$ is a clique with $\geq 3$ vertices, we can assume, without loss of generality, that $G$ is precolored by coloring $X$ red. Let $C$ be a component of $G - X$. Since $C$ is dominated by $X$, any good coloring of $C$ must be monochromatic. Altogether, this implies that if $G$ has a strong coloring, then $G$ has a strong coloring such that at least one component of $G - X$ is blue. As this can clearly be checked in polynomial-time, the proof is complete.  
\end{proof}

\section{The proof of Theorem \ref{thm:p12}}

An instance $(X, \mathcal{C})$ of \emph{Restricted Positive 1-in-3-SAT} consists of a set of Boolean variables $X = \{x_1, \dots, x_n\}$ and a collection of clauses $\mathcal{C} = \{C_1, \dots, C_m\}$, where each clause is a disjunction of exactly three variables, and the question is to determine whether there exists a satisfying truth assignment so that exactly one variable in each clause is set to true. This problem is NP-complete \cite{schmidt2010computational}.

\begin{proof}[Proof of Theorem \ref{thm:p12}]
The problem \textsc{matching-cut} is clearly in NP.  Let $F = (X, \mathcal{C})$ be any instance of Restricted Positive 1-in-3 SAT. We construct, in polynomial time, a graph $G$ that is $P_k$-free for some positive integer $k$ such that $F$ is satisfiable iff $G$ has a matching-cut. 
 
From now on, we fix a variable $s \in X$. For each variable $x \in X \setminus \{s\}$, we build a variable gadget depicted in Figure \ref{fig:var}, where each  of the two circles $C_x^v$ and $C_x^u$ depicts a clique on five vertices. Similarly, for $s$ we build a variable gadget as in Figure \ref{fig:var}, except that each of the circles $C_s^v$ and $C_s^u$ depicts a clique on $|X| + 3$ vertices, where vertices of the left circle are labelled $v_{x_1}, v_{x_2}, v_{x_3}, v_x, v_x^1, \dots, v_x^{|X|-1}$ and vertices of the right circle  $u_{x_1}, u_{x_2}, u_{x_3}, u_x, u_x^1, \dots, u_x^{|X|-1}$.  

For $i \in \{1, 2, 3\}$, we think of $v_{x_i}$ as \emph{corresponding} to the $i$th occurrence of $x$ and, as will become evident by the end of the proof, of $u_{x_i}$ as ``complementary"  to $v_{x_i}$; we also call $v_{x_i}$ and $u_{x_i}$ \emph{variable vertices}.    We connect the set of variable gadgets as follows. Set $V_s = \{ v_s^1, \dots, v_s^{|X|-1}\}$, $U_s = \{u_s^1, \dots, u_s^{|X|-1} \}$, $W = \{v_x^1:x\in X \setminus \{s\} \}$ and $Z =\{u_x^1:x\in X \setminus \{s\} \}$, let $f : V_s \rightarrow W$ and $g : U_s \rightarrow Z$ be bijective so that, furthermore, $f(v_s^j)$ and $g(u_s^j)$ are members of the the same variable gadget for $j \in \{1, \dots, |X| - 1\}$, and add the edges $v_s^jf(v_s^j)$, $v_s^jg(u_s^j)$,  $u_s^jg(u_s^j)$ and $u_s^jf(v_s^j)$ for $j \in \{1, \dots, |X| - 1\}$. See Figure \ref{fig:fg} for an example illustrating the edges between the variable gadgets corresponding to $s$, $x$ and $y$.  

For each clause $C \in \mathcal{C}$, we build a clause gadget depicted in Figure \ref{fig:clause}, where vertices $v_C, u_C^1, u_C^2, u_{x_1, C, 1}, u_{x_1, C, 2}, u_{z_2, C, 1}, u_{z_2, C, 2}, u_{y_3, C, 1}, u_{y_3, C, 2}$ are new vertices, and vertices $v_{x_1}, v_{z_2}, v_{y_3}, u_{x_1}, u_{z_2}, u_{y_3}$ are variable vertices to be found in, respectively, $C_x^v, C_z^v, C_y^v, C_x^u, C_z^u$ and $C_y^u$. We call the vertices $v_C$ and $u^i_C$ for $ i \in \{1,2\}$ \emph{special}.   We complete the construction of $G$ by adding an edge between every pair of special vertices.

 \begin{figure}
\begin{center}\includegraphics[width=0.8\textwidth]{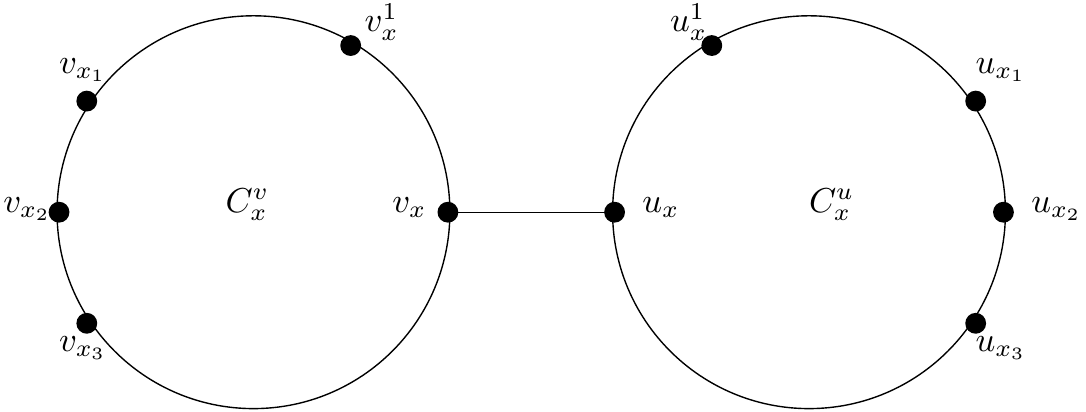}\end{center}
\caption{Variable gadget}\label{fig:var}
\end{figure}

 \begin{figure}
\begin{center}\includegraphics[width=0.8\textwidth]{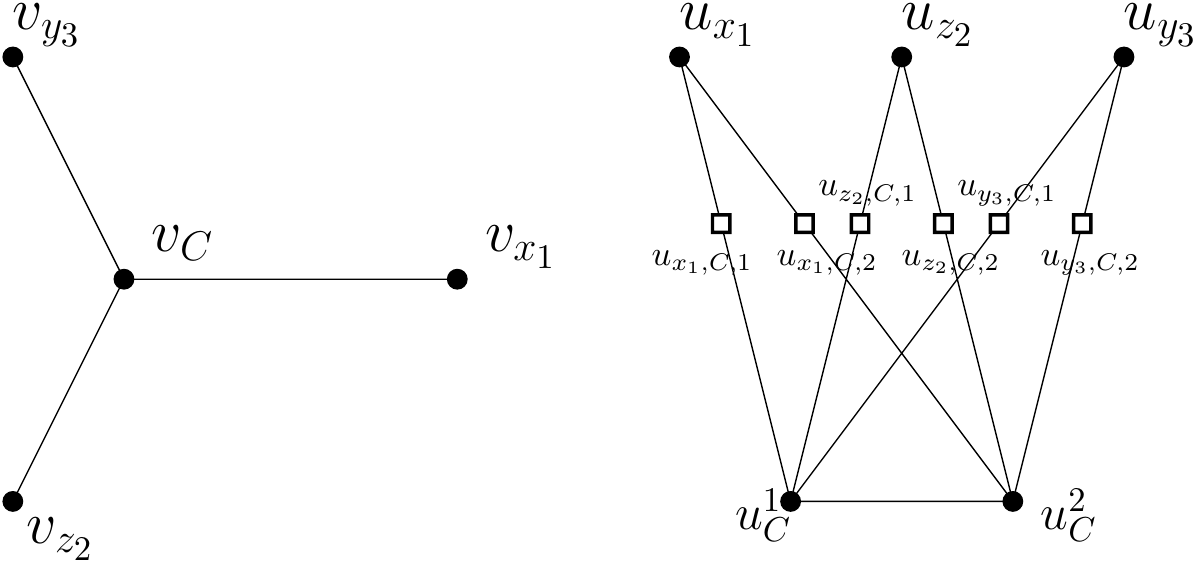}\end{center}
\caption{Clause gadget for $C = (x_1 \lor y_3 \lor z_2)$}\label{fig:clause}
\end{figure}

 \begin{figure}
\begin{center}\includegraphics[width=1\textwidth]{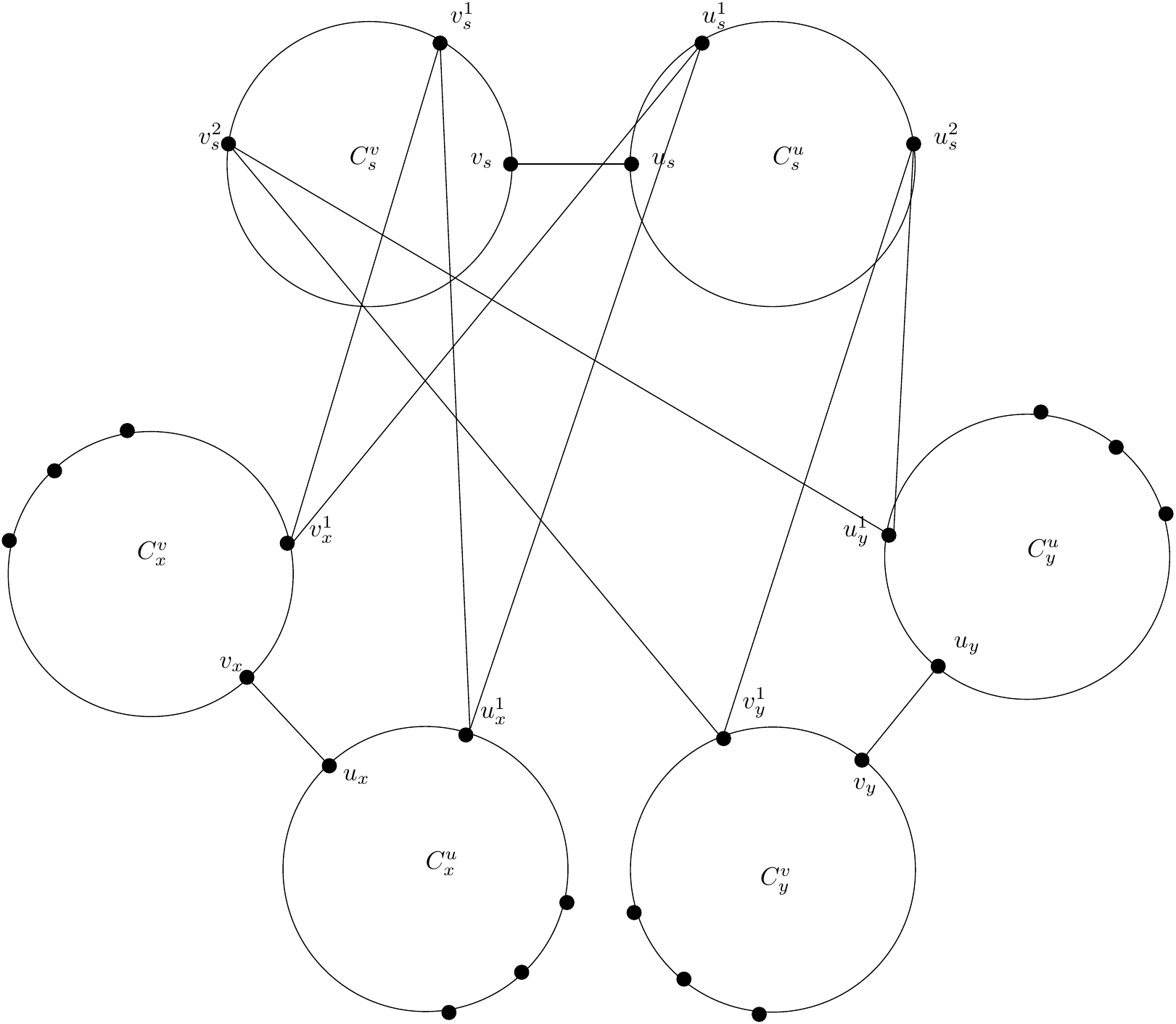}\end{center}
\caption{Edges between three variable gadgets corresponding to variables $s$, $x$ and $y$.}\label{fig:fg}
\end{figure}

Recall from Section \ref{sec:1} that $G$ has a strong coloring if and only if $G$ has a matching-cut. 

Suppose $G$ has a strong coloring $\varphi$. 

\begin{claim}\label{special}
The set of special vertices and, for each $x \in X$, the circles $C_x^v$ and $C_x^u$ are each monochromatic. 
\end{claim}

\begin{proof}
Immediate from the fact that any complete graph on at least three vertices must be monochromatic in a strong coloring. 
\end{proof}

Call (the strong coloring) $\varphi$ \emph{$S$-splitting} for some $S \subset X$ if $\varphi(C_x^v) \not= \varphi(C_x^u)$ for each $x \in S$. 

\begin{claim}\label{claim:bi}
Given a clause $C = (x \lor y \lor z)$, if $\varphi$ is $\{x, y, z \}$-splitting, then $\{v_x, v_y, v_z\}$ is bichromatic. 
\end{claim}

\begin{proof}
Suppose for a contradiction that $\{v_x, v_y, v_z\}$ is monochromatic and assume, without loss of generality, that its color is red. Then $v_C$ is also red, since otherwise $\varphi$ is not strong.

On the other hand, since $\varphi$ is $\{x, y, z \}$-splitting, the color of  $\{u_x, u_y, u_z\}$ is blue, which in turn implies that at least three of the vertices in  $\bigcup_{t \in \{x, y, z\}, i \in \{1, 2\} }\{u_{t, C, i}\}$ are blue and so at least one of $u^1_C$, $u_C^2$ is also blue. This contradicts Claim \ref{special}.  
\end{proof}

Recall that we abbreviate red and blue by $r$ and $b$, respectively.  

\begin{claim}\label{claim:equal}
Given clauses $C = (x \lor y \lor z)$ and $C' = (p \lor q \lor t)$, if $\varphi$ is $\{x, y, z, p, q, t \}$-splitting, then $$|\varphi(\{v_x, v_y, v_z \})\cap \{r\}| = |\varphi(\{v_p, v_q, v_t \})\cap \{r\}| \in \{1,2 \}.$$ 
\end{claim}

\begin{proof}
By Claim \ref{claim:bi}, $|\varphi(\{v_x, v_y, v_z \})\cap \{r\}|,  |\varphi(\{v_p, v_q, v_t \})\cap \{r\}| \in \{1,2 \}$. If for a contradiction $1 = |\varphi(\{v_x, v_y, v_z \})\cap \{r\}| < |\varphi(\{v_p, v_q, v_t \})\cap \{r\}| = 2$, then by construction $v_C$ is blue and $v_{C'}$ is red, which contradicts Claim \ref{special}. 
\end{proof}

\begin{claim}\label{claim:splitx}
If $\varphi$ is  $\{x\}$-splitting for some $x \in X$, then $\varphi$ is $X$-splitting. 
\end{claim}

\begin{proof}

We distinguish two cases. 

Suppose first that $\varphi$ is $\{s\}$-splitting. For each $j \in \{1, \dots, |X| - 1\}$, since $v_s^jf(v_s^j)$, $v_s^jg(u_s^j)$,  $u_s^jg(u_s^j)$ and $u_s^jf(v_s^j)$ are edges and since, by assumption, $v_s^j$ and $u_s^j$ differ in color, $f(v_s^j)$ and $g(u_s^j)$ must also differ in color (else $\varphi$ is not good). Thus, $\varphi$ is $X$-splitting; see Figure \ref{fig:claim} for an illustration. 

In all other cases, $\varphi$ is $\{x\}$-splitting for some $x \in X \setminus \{s\}$. Then an analogous argument implies  $\varphi$ is $s$-splitting which in turn implies the claim. 
\end{proof}

\begin{claim}\label{claim:last}
$\varphi$ is  $X$-splitting. 
\end{claim}

\begin{proof}
Otherwise, by Claim  \ref{claim:splitx} and its proof, the graph induced by the union of the variable gadgets is monochromatic, say has color red. This in turn implies, by construction and Claim \ref{special}, that every special vertex is red. On the other hand, by definition, $G$ has a blue vertex. But this vertex cannot be a non-special vertex of a clause gadget else it would have both neighbors red. Therefore, $G$ is red itself, a contradiction.  
\end{proof}

We are now ready to show that $F$ is satisfiable. Since $\varphi$ is $X$-splitting by Claim \ref{claim:last}, we can assume, by Claim \ref{claim:equal} and interchanging the roles of red and blue if necessary, that $|\varphi(\{v_x, v_y, v_z \})\cap \{r\}| = 1$ for each clause $(x \lor y \lor z) \in \mathcal{C}$. We now set a variable $x \in X$ to true if and only if its corresponding vertices are red. 

Conversely, suppose $F$ is satisfiable.  For each variable $x \in X$, we give the vertices in $C^v_x$ color red if $x$ is set to true and blue otherwise. We extend this partial coloring to an $X$-splitting coloring of the graph induced by the union of the variable gadgets. We complete this partial coloring to a coloring of $G$ by coloring each special vertex with color blue and, for each clause gadget corresponding to a clause, say $C = (x, y, z)$, assuming without loss of generality $u_x$ is blue and $u_y$ and $u_z$ are red, color blue $u_{x, C, 1}, u_{x, C, 2}, u_{y, C, 1}, u_{z, C, 2}$ and red $u_{y, C, 2}, u_{z, C, 1}$. To see that the resulting coloring is strong, it suffices to argue that the coloring restricted to any clause gadget $C = (x, y, z)$ is strong (since, by assumption, the coloring is $X$-splitting). As $v_C$ is blue, $v_x$ is red and $v_y, v_z$ are blue, each of these vertices has at most one neighbor of the other color and so the coloring restricted to the graph induced by $\{v_C, v_x, v_y, v_z\}$ is strong. Similarly, as $u_C^1$ and $u_C^2$ are blue, $u_{x, C, 1}, u_{x, C, 2}, u_{y, C, 1}, u_{z, C, 2}$ are blue, $u_{y, C, 2}, u_{z, C, 1}$ are red, $u_x$ is blue and $u_y$ and $u_z$ are red, each of these vertices has at most one neighbor of the other color and so we are done.  

To complete the proof, it remains to show that $G$ is $P_{k}$-free for some $k > 0$. Suppose $G$ contains an induced path $P$ with $t \geq 1$ vertices.  Since the set of special vertices induces a complete graph, $P$ can contain at most two special vertices and these are consecutive on $P$.  Similarly, by construction, $P$ can contain at most four vertices from each variable gadget.  How many variable gadgets can have vertices in common with $P$? 

Any two variable gadgets are connected either via a special vertex or via the variable gadget of $s$ and therefore, by our earlier observations, the number of such variable gadgets is bounded;  as $t$ can contain at most two special vertices and at most four vertices from a bounded number of variable gadgets, $t$ is also bounded.  This completes the proof.   
\end{proof}

We should remark that the same proof works via a reduction from the more well-known Positive 1-in-3-SAT problem, that is, the 1-in-3-SAT problem in which every variable occurs as positive (but may also appear more than three times). We chose the restricted version of this problem for ease of presentation.

\subsection*{Acknowledgements}
I thank both referees for very carefully reading the paper and for their feedback that significantly improved the presentation. I also thank Pierre Berg\'e for a very helpful discussion and Valentin Bouquet and Van Bang Le for their encouragements and telling me about \cite{chvatal}.
This work was supported by
the French National Research Agency under research grant ANR DIGRAPHS ANR-19-CE48-0013-01.

 \bibliography{bibliography}{}
\bibliographystyle{abbrv}

\end{document}